\documentclass[11pt,reqno]{amsart}
\usepackage{amsmath}
\usepackage{amssymb}
\usepackage{amsthm}
\usepackage{romannum}
\usepackage{enumitem}

\usepackage{scalerel,stackengine}
\stackMath
\newcommand\reallywidehat[1]{%
	\savestack{\tmpbox}{\stretchto{%
			\scaleto{%
				\scalerel*[\widthof{\ensuremath{#1}}]{\kern.1pt\mathchar"0362\kern.1pt}%
				{\rule{0ex}{\textheight}}
			}{\textheight}%
		}{2.4ex}}%
	\stackon[-6.9pt]{#1}{\tmpbox}%
}

\newlength{\defbaselineskip}
\newcommand{\setlinespacing}[1]%
{\setlength{\baselineskip}{#1 \defbaselineskip}}

\numberwithin{equation}{section}

\newtheorem{thm}{Theorem}[section]

\newtheorem{lem}[thm]{Lemma}
\newtheorem{prop}[thm]{Proposition}

\theoremstyle{definition}

\theoremstyle{remark}
\newtheorem{rem}[thm]{Remark}
\numberwithin{equation}{section}

\newcommand{\lgl}{\langle}
\newcommand{\rgl}{\rangle}

\makeatletter
\@namedef{subjclassname@2020}{\textup{2020} Mathematics Subject Classification}
\makeatother

\begin{document}
	\pagenumbering{arabic}\setcounter{page}{1}

\title[Talbot effect]
{Talbot effect for the third order Lugiato-Lefever equation}
	
\author[Gunwoo Cho, Seongyeon Kim and Ihyeok Seo]{Gunwoo Cho, Seongyeon Kim and Ihyeok Seo}


\address{Department of Mathematics, Sungkyunkwan University, Suwon 16419, Republic of Korea}
\email{ihseo@skku.edu}
\email{gwcho@skku.edu}
 
\address{Department of Mathematics Education, Jeonju University, Jeonju 55069, Republic of Korea}
\email{sy\_kim@jj.ac.kr}

\begin{abstract}
We discuss the Lugiato-Lefever equation and its variant with third-order dispersion, which are mathematical models used to describe how a light beam forms patterns within an optical cavity. It is mathematically demonstrated that the solutions of these equations follow the Talbot effect, which is a phenomenon of periodic self-imaging of an object under certain conditions of diffraction. The Talbot effect is regarded as the underlying cause of pattern formation in optical cavities.
\end{abstract}
	
\maketitle
	
\section{Introduction}\label{sec1}
In 1987, L. Lugiato and R. Lefever \cite{LL} discovered an intriguing behavior of light beams in an optical cavity that involves the interplay between dispersion and diffraction. Specifically, they observed that the beam forms stationary hexagonal patterns in the transverse direction, while exhibiting Kerr frequency combs along the longitudinal direction. This phenomenon can be described by the Lugiato-Lefever equation,
	\begin{equation}\label{SLL}
		\begin{cases}
			\partial_t u=i\alpha\partial_x^2u-(1+i\theta)u+i|u|^2u+u_0,\\
			u(0,x)=u_0(x),	
		\end{cases}
	\end{equation} 
	where $\alpha$ and $\theta$ are real parameters.
To simulate the pattern formation predicted by this equation, various methods have been proposed, including those presented in \cite{GC} and \cite{SFTLL}.

The formation of patterns is closely connected to the Talbot effect, which arises from the diffraction of light waves (see, for example, \cite[Chap. 3]{VM}). Discovered by H. F. Talbot \cite{T1} in 1836, the effect occurs when white light passes through a diffraction grating, resulting in the recovery of the grating pattern at regular intervals known as the Talbot distance ($d_T$). This distance can be calculated using the formula $d_\mathrm{T} = a^2/\lambda$ by Rayleigh \cite{R1}, where $a$ is the grating spacing and $\lambda$ is the wavelength of the light. At rational multiples of $d_T$, the pattern appears as a finite linear combination of the grating pattern. In this study, we aim to mathematically demonstrate the Talbot effect for the Lugiato-Lefever equation (equation \eqref{TLL}) with third-order dispersion, as well as for equation \eqref{SLL}.

The Talbot effect has been the subject of study in a series of papers by Berry and his collaborators \cite{B,BK,BL,BMS}. In particular, in \cite{BK}, they considered the linear Schr\"odinger equation and demonstrated that the solution at rational times $t\in d_T\mathbb{Q}$ is a linear combination of a finite number of translates of its initial data. This phenomenon is sometimes referred to as \textit{dispersive quantization} in the literature. They also showed that the solution at irrational times $t\notin d_T\mathbb{Q}$ has a fractal profile that is nowhere differentiable.

More rigorous mathematical works on this topic were initiated by Oskolkov \cite{O}. He studied a large class of linear dispersive equations on $\mathbb{T}=\mathbb{R}/2\pi\mathbb{Z}$ with initial data of bounded variation. The solution is a continuous function that is nowhere differentiable if $t\notin \pi \mathbb Q$. Furthermore, if the initial data is continuous, then the solution in space-time is also continuous. However, if $t\in\pi \mathbb Q$ and the initial data contains discontinuities, the solution is necessarily discontinuous. This was further investigated in \cite{KR,R2,T3}, where Rodnianski \cite{R2} proved that the graphs of the real and imaginary parts of the solution to the linear Schr\"odinger equation have an upper Minkowski dimension of $3/2$ at almost all irrational times $t\notin \pi \mathbb Q$.

In \cite{ZWZX}, the Talbot effect and fractality were observed experimentally in a nonlinear setting, while in \cite{ch-ol14}, they were observed numerically. In \cite{ET1,CET,ES}, some of these effects were rigorously studied, with a key ingredient being a smoothing estimate on the Bourgain spaces for the nonlinearity.

The Lugiato-Lefever equation \eqref{SLL} can be seen as a damped and forced cubic nonlinear Schrödinger equation, but obtaining such estimates becomes more cumbersome due to the forcing term that appears in our case
(see Section \ref{sec3}).
More generally, we are concerned here with the Lugiato-Lefever equation with third-order dispersion:
\begin{equation}\label{TLL}
		\begin{cases}
			\partial_t u=\beta\partial_x^3 u+i\partial_x^2u-(1+i\theta)u+i|u|^2u+u_0,\\
			u(0,x)=u_0(x),
		\end{cases}	
	\end{equation}
where $x\in\mathbb{T}=\mathbb{R}/2\pi\mathbb{Z}$, $t\in\mathbb{R}$, and $\beta,\theta\in\mathbb{R}$. 
We always assume $\beta\in\mathbb{Z}$ technically, and $\theta=0$ just for convenience.
It is also worth noting that the effects of third-order dispersion are receiving increasing attention for the study of cavity solitons associated with Kerr frequency combs \cite{P}.

Our results can be summarized in two parts. Firstly, in Theorem \ref{Tal}, we establish dispersive quantization. The key observation here is that the behavior of the solution to \eqref{TLL} differs significantly between rational and irrational times, depending on the Talbot distance $d_T=\pi$ in our setting. In other words, the solution does not distribute uniformly but rather forms quanta. Secondly, in Theorem \ref{t:dim}, we demonstrate that the solution exhibits fractal nature at irrational times, quantified by calculating the upper Minkowski dimension of the graph of the solution.
	
	\begin{thm}\label{Tal}
Let $u$ be a solution to \eqref{TLL} with initial data $u_0\in BV$.
If $t\notin\pi\mathbb{Q}$, then the solution $u(t,x)$ is a continuous function of $x$.
On the other hand, if $t\in\pi\mathbb{Q}$ and $u_0$ has at least one discontinuity on $\mathbb T$, then $u(t,x)$ is a bounded function that necessarily contains at most countably many discontinuities.
However, if $u_0$ is continuous, then $u(t,x)$ is jointly continuous in both temporal and spatial variables. 
	\end{thm}
	
	\begin{thm}\label{t:dim}
  Let $u$ be a solution to \eqref{TLL} with initial data $u_0\in BV$. Suppose that
$$\sigma_0:= \sup\{\sigma\in \mathbb R\colon u_0\in H^\sigma\}<5/8.$$
(Recall that if $u_0\in BV$, then $u_0\in H^{\frac12+}$, and so $\sigma_0\ge 1/2$.)
Then, for almost all $t\notin\pi \mathbb Q$, the upper Minkowski dimension of the graphs of $\mathrm{Re}u(t,\cdot)$ and $\mathrm{Im}u(t,\cdot)$ lies in the interval $\left[\frac94-2\sigma_0, \frac74\right]$.
	\end{thm}

The Talbot effect for the Lugiato-Lefever equations suggests an interesting research direction involving numerical investigations of the Talbot effect in related optical systems. For the reader's convenience, we refer to several numerical studies \cite{IMCSMS, ZZJL, ZDJZL, ZDKL} on optical systems. 
    
	\subsubsection*{Organization} 
 Firstly, in Section \ref{sec2}, we establish the local well-posedness of the Cauchy problem \eqref{TLL} in the space $H^s(\mathbb{T})$. In Section \ref{sec3-0}, we present the key smoothing estimate (Proposition \ref{smoothing}) for the Duhamel part of problem \eqref{TLL}, along with the relevant results on the free evolution part that we require. We then prove Theorems \ref{Tal} and \ref{t:dim}. The remaining sections are dedicated to proving the smoothing estimate.

	\subsubsection*{Notation}
We use the symbol $C$ to denote a positive constant, which may differ from one occurrence to another. Given $A, B > 0$, we write $A\lesssim B$ if $A\leq CB$, where $C>0$ is some constant. We also use the notation $A\approx B$ if $A\lesssim B$ and $B\lesssim A$. For a function $f$ defined on $\mathbb{T}=\mathbb{R}/2\pi\mathbb{Z}$, we define its Fourier transform by
\begin{equation*}
\widehat{f}(k) = \int_{0}^{2\pi} e^{-ikx}f(x) dx, \quad k\in\mathbb{Z},
\end{equation*}
and its space-time Fourier transform by
\begin{equation*}
\widetilde{f}(\tau,k) = \int_{\mathbb{R}} \int_{0}^{2\pi} e^{-i(\tau t+kx)} f(t,x) dx dt, \quad (\tau, k)\in \mathbb{R}\times \mathbb{Z}.
\end{equation*}
For every $s\geq 0$, we denote by $H^s$ the Sobolev space on $\mathbb{T}=\mathbb{R}/2\pi\mathbb{Z}$ equipped with the norm
\begin{equation*}
		\|f\|_{H^s} := \Big( \sum_{k\in\mathbb Z} \langle k\rangle^{2s} |\widehat f (k)|^2 \Big)^{1/2},
	\end{equation*}
where $\langle k\rangle=(1+|k|^2)^{1/2}$.

	\section{Local well-posedness}\label{sec2}
	In this section we establish a local well-posedness for the Cauchy problem \eqref{TLL} in $H^s(\mathbb{T})$.
	We first introduce the Bourgain space $X^{s,b}$, $s,b\in\mathbb{R}$, equipped with the norm $$\|f\|_{X^{s,b}(\mathbb{R}\times\mathbb{T})}:=\|\langle k\rangle^{s}\langle \tau+k^3+ k^2\rangle^{b}\widetilde{f}(\tau,k)\|_{L_\tau^2 l_k^2(\mathbb{R}\times\mathbb{Z})},$$
	and its restricted space $X^{s,b}_{\delta}$, $0<\delta\leq1$, with the corresponding norm
	\[	\|u\|_{X^{s,b}_{\delta}} = \inf\big\{\|w\|_{X^{s,b}} \colon w\vert_{[-\delta,\delta]\times\mathbb T}=u \big\}.\]
	The local well-posedness we will obtain is now stated as follows.
	
	\begin{thm}\label{wellp}
		Let $s\geq0$ and $1/2<b<5/8$. If $u_0\in H^s(\mathbb{T})$, then there exist $\delta>0$ and a unique solution $u \in C_t([-2\delta,2\delta];H_x^s(\mathbb{T}))\cap X_{2\delta}^{s,b}(\mathbb{R}\times\mathbb{T})$ with 
		\begin{equation}\label{prop2}
			\|u\|_{X_{2\delta}^{s,b}}\lesssim \|u_0\|_{H^{s}}.
		\end{equation}
	\end{thm}

	\begin{proof}
		By Duhamel's principle, we write the solution to the Cauchy problem \eqref{TLL} as  
		\begin{equation}\label{sol}
			\Phi(u):= e^{(\partial_x^3+i\partial_x^2-1)t}u_0+i\int_0^t e^{(\partial_x^3+i\partial_x^2-1)(t-t')}F(u)(t',\cdot)\,dt'
		\end{equation} 
		where $F(u)=i|u|^2u+u_0$.
		For some small $\delta>0$ and large $K>0$ which will be chosen later, we will show that $\Phi$ defines a contraction map on the set $$X:=\{u\in X_{2\delta}^{s,b}:\|u\|_{X_{2\delta}^{s,b}}\leq K\|u_0\|_{H^s}\}.$$
		To do so, we first present some basic properties of the Bourgain space, Lemmas \ref{Bour-Str} and \ref{lem2}. 
		The former can be found in \cite[Proposition 2.12]{T2} and \cite[Lemma 3.3]{ET2}. See also \cite[Lemma 2.11]{T2} for the latter.
		
		\begin{lem}\label{Bour-Str}
			Let $s\in\mathbb{R}$ and $1/2<b\leq 1$. 
			If we set $\eta$ to be a smooth function supported on $[-2,2]$ with $\eta=1$ on $[-1,1]$, then 
			\begin{gather}
				\| e^{(\partial_x^3+i\partial_x^2-1)t}f\|_{X^{s,b}_{\delta}} \leq C\|f\|_{H^{s}},\label{homo} \\ \left\|\eta(t)\int_{0}^t e^{(\partial_x^3+i\partial_x^2-1)(t-t')}F(t',\cdot)\, dt'\right\|_{X_{\delta}^{s,b}}\le C\|F\|_{X_{\delta}^{s,b-1}}\label{inhomo}
			\end{gather}	
			with a constant $C$ depending only on $b$.
		\end{lem}
		\begin{lem}\label{lem2}
			Let $s\in\mathbb{R}$ and $-1/2<b<b^\prime<1/2$. Then 
			\[	\|u\|_{X^{s,b}_{\delta}} \leq C\delta^{b^\prime-b}\|u\|_{X^{s,b^\prime}_{\delta}},\] with a constant $C$ depending only on $b$ and $b'$.
		\end{lem}
		
		Now we show that $\Phi$ is well-defined on $X$.
		By applying \eqref{homo} and \eqref{inhomo} to the homogeneous and Duhamel terms in \eqref{sol}, respectively, we have
		\begin{equation*}
			\|e^{(\partial_x^3+i\partial_x^2-1)t}u_0\|_{X_{2\delta}^{s,b}} \leq C\|u_0\|_{H^s}
		\end{equation*}
		and
		\begin{equation}\label{last}
			\Big\|\int_0^t e^{(\partial_x^3+i\partial_x^2-1)(t-t')}F(u)(\cdot,t')\,dt'\Big\|_{X_{2\delta}^{s,b}}
			\le C\big(\|u_0\|_{X_{2\delta}^{s,b-1}}+\||u|^2u\|_{X_{2\delta}^{s,b-1}}\big)
		\end{equation}
		if $1/2<b\leq1$.
		Using Lemma \ref{lem2}, the left-hand side of \eqref{last} is also bounded by
		\begin{equation}\label{last1}
			C\delta^{b'-b}(\|u_0\|_{X_{2\delta}^{s,b'-1}}+\||u|^2u\|_{X_{2\delta}^{s,b'-1}})
		\end{equation}
		whenever $1/2<b\leq1$ and $b<b'<\frac32$.
		To bound \eqref{last1} again with $C\|u_0\|_{H^s}$, we make use of the following proposition which will be proved at the end of this section:
		\begin{prop}\label{prop3} Let $s\geq0$. 
			If $0\leq b'\leq1$,
			then
			\begin{align}\label{force}
				\|u_0\|_{X_\delta^{s,b'-1}}&\lesssim\|u_0\|_{H^s},
			\end{align}
			and if $1/3<b'\leq5/8$
			\begin{align}\label{trilinear}
				\|uvw\|_{X_\delta^{s,b'-1}}\lesssim\|u\|_{X_\delta^{s,\frac38}}\|v\|_{X_\delta^{s,\frac38}}\|w\|_{X_\delta^{s,\frac38}}.
			\end{align}
			Here the implicit constants depend only on $b'$.
		\end{prop}
		Since we are assuming $1/2<b<5/8$ in Theorem \ref{wellp}, by applying Proposition \ref{prop3} to \eqref{last1}, we get 
		\begin{align*}	
			C\delta^{b'-b}(\|u_0\|_{X_{2\delta}^{s,b'-1}}+\||u|^2u\|_{X_{2\delta}^{s,b'-1}})
			&\leq C\delta^{b'-b}(\|u_0\|_{H^s} + \|u\|^3_{X_{2\delta}^{s,\frac38}})\\
			&\leq C\delta^{b'-b}(\|u_0\|_{H^s} + K^3\|u_0\|_{H^s}^3)
		\end{align*}
		where we used for the last inequality that $$\|u\|_{X_{2\delta}^{s,\frac38}}\leq\|u\|_{X_{2\delta}^{s,b}}\leq K\|u_0\|_{H^s}$$ 
		since $u\in X$.
		If we take $\delta$ small so that
		\begin{equation*}
			\delta^{b'-b}\leq \frac{1}{1+K^3\|u_0\|_{H^s}^2},
		\end{equation*}
		then finally
		\begin{equation}\label{bd}
			\|\Phi(u)\|_{X_{2\delta}^{s,b}}\le C\|u_0\|_{H^s}.
		\end{equation}
		Choosing $K$ bigger than $C$ here, we conclude 
		$\Phi(u)\in X_{2\delta}^{s,b}$ for $u\in X_{2\delta}^{s,b}$.
		
		Next we show that $\Phi$ is a contraction on $X$. Namely, if $u,v\in X$,
		\begin{equation} \label{con}
			\|\Phi(u)-\Phi(v)\|_{X_{2\delta}^{s,b}}\leq \frac{1}{2}\|u-v\|_{X_{2\delta}^{s,b}}.
		\end{equation}
		Using \eqref{inhomo} and Lemma \ref{lem2}, we have for $\frac12<b<b'<\frac32$ 
		\begin{align*}
			\|\Phi(u)-\Phi(v)\|_{X_{2\delta}^{s,b}}&\le C\left\||u|^2u-|v|^2v\right\|_{X_{2\delta}^{s,b-1}}\\
			&\le C\delta^{b'-b}\left\||u|^2u-|v|^2v\right\|_{X_{2\delta}^{s,b'-1}}.
		\end{align*}
		By applying \eqref{trilinear} here after using the following simple inequality 
		\begin{equation*}
			\big||u|^2u-|v|^2v\big|\lesssim \big(|u|^2+|v|^2\big)|u-v|,
		\end{equation*}
		we see that
		\begin{align*}
			C \delta^{b'-b}\left\||u|^2u-|v|^2v\right\|_{X_{2\delta}^{s,b'-1}}&\le C  \delta^{b'-b} \left\|\big(|u|^2+|v|^2\big)|u-v|\right\|_{X_{2\delta}^{s,b'-1}}\\
			&\leq C \delta^{b'-b}\big(\|u\|_{X_{2\delta}^{s,\frac38}}^2+\|v\|_{X_{2\delta}^{s,\frac38}}^2\big)\|u-v\|_{X_{2\delta}^{s,\frac38}}.
		\end{align*}
		Hence, we get
		\begin{align*}
			\|\Phi(u)-\Phi(v)\|_{X_{2\delta}^{s,b}}&\le C\delta^{b'-b} \big(\|u\|_{X_{2\delta}^{s,\frac38}}^2+\|v\|_{X_{2\delta}^{s,\frac38}}^2\big)\|u-v\|_{X_{2\delta}^{s,\frac38}}\\
			&\le C\delta^{b'-b}K^2\|u_0\|_{H^s}^2\|u-v\|_{X_{2\delta}^{s,b}}\\
			&\le \frac{CK^2\|u_0\|_{H^s}^2}{(1+K^3\|u_0\|_{H^s}^2)}\|u-v\|_{X_{2\delta}^{s,b}}.
		\end{align*}
		This implies \eqref{con} if we choose $K$ bigger than $2C$ here.
		
		Therefore, the local existence in Theorem \ref{wellp} follows by the contraction mapping principle and the embedding $X_{2\delta}^{s,b} \hookrightarrow C([-2\delta,2\delta];H^s)$.
		The inequality \eqref{prop2} also holds immediately from \eqref{bd}.
	\end{proof}
	
	\subsection*{Proof of Proposition \ref{prop3}}
	It remains to prove Proposition \ref{prop3}.
	We first recall the following lemma which will be also used several times later:
	\begin{lem}[\cite{ET3}, Lemma 3.3]\label{lem3}
		If $\beta \geq \gamma \geq 0$ and $\beta +\gamma>1$, then
		\begin{align*}
			\sum_n \frac{1}{\langle n-k_1\rangle^{\beta}\langle n-k_2\rangle^{\gamma}} 
			\approx \int_{\mathbb{R}} \frac{1}{\langle \tau-k_1\rangle^{\beta}\langle \tau-k_2\rangle^{\gamma}} d\tau 
			\lesssim \frac{\phi_{\beta}(k_1-k_2)}{\langle k_1-k_2\rangle^{\gamma} }
		\end{align*}
		where
		\begin{equation*}
			\phi_{\beta}(k):= \sum_{|n|\leq|k|}\frac{1}{\langle n\rangle^{\beta}} \approx
			\begin{cases}
				1,&\beta>1,\\
				\log(1+\langle k\rangle),&\beta=1,\\
				\langle k\rangle^{1-\beta},&0\le\beta<1.
			\end{cases}
		\end{equation*}
	\end{lem}
	
	Let us now prove \eqref{force} first.
	Let $\eta\in C_0^{\infty}(\mathbb{R})$ be a smooth function supported on $[-2,2]$ with $\eta=1$ on $[-1,1]$.
	By the definition of $X_\delta^{s,b}$-norm, we see
	\begin{align}
		\nonumber
		\|u_0\|_{X_\delta^{s,b'-1}}^2&\le\|\eta(t) u_0\|_{X^{s,b'-1}}^2\\
		\nonumber
		&=\sum_k\int_\tau\langle k\rangle^{2s}\langle \tau+k^3+k^2\rangle^{2b'-2}|\hat{u_0}(k)|^2|\hat{\eta}(\tau)|^2d\tau\\
		\label{a}
		&\lesssim\sum_k\langle k\rangle^{2s}|\hat{u_0}(k)|^2\int_\tau\langle \tau+k^3+k^2\rangle^{2b'-2}\langle\tau\rangle^{-2}d\tau.
	\end{align}
	Here, for the last inequality, we used the fact that if $\eta$ is a compactly supported smooth function then $\hat\eta$ decays faster than $\langle \tau\rangle^{-p}$ for any $p\in\mathbb{N}$. (It is sufficient to choose $p$ as 2.)
	Since $0\leq b'\leq 1$, applying Lemma \ref{lem3} with $\beta=2, \gamma=2-2b'$ to the integration in $\tau$ and using $\langle k^3+k^2\rangle\sim\langle k^3\rangle\sim\langle k\rangle^3$ in turn, the right-hand side of \eqref{a} is bounded by 
	\begin{equation*}
		\sum_k \langle k\rangle^{2s}\langle k^3+ k^2\rangle^{2b'-2}|\hat{u_0}(k)|^2\sim\sum_k\langle k \rangle^{2s+6b'-6}|\hat{u_0}(k)|^2\lesssim\|u_0\|_{H^s}^2.
	\end{equation*}
	
	Next we prove \eqref{trilinear}.
	Since $X^{s,-\frac38} \subset X^{s,b'-1}$ for $b'\leq 5/8$,
	by duality it is sufficient to prove that 
	\begin{equation*}
		\left|\langle uvw, z\rangle_{L^2(\mathbb{T}\times\mathbb{R})}\right| \lesssim\|u\|_{X^{s,\frac38}}\|v\|_{X^{s,\frac38}}\|w\|_{X^{s,\frac38}} \|z\|_{X^{-s,\frac{3}{8}}}.
	\end{equation*}
	Let us set $\mathcal J^s$ to be the Fourier multiplier operators defined by the multiplier $\langle k\rangle^s$.
	Then by H\"older's inequality, 
	\begin{align}
		\nonumber
		\big|\langle uvw, z\rangle_{L^2(\mathbb{T}\times\mathbb{R})}\big|&=\Big|\left\langle \mathcal J^s(uvw), \mathcal J^{-s}z\right\rangle_{L^2(\mathbb{T}\times\mathbb{R})}\Big|\\
		\label{c}
		&\le\left\|\mathcal J^s(uvw)\right\|_{L^{\frac43}(\mathbb{T}\times\mathbb R)}\|\mathcal J^{-s}z\|_{L^4(\mathbb T\times\mathbb R)}.
	\end{align}
	To bound the term $\left\|\mathcal J^s(uvw)\right\|_{L^{\frac43}(\mathbb{T}\times\mathbb R)}$, we apply the following lemma (Lemma \ref{Js}) twice as
	\begin{align}
		\nonumber
		\left\|\mathcal J^s(uvw)\right\|_{L^{\frac43}(\mathbb{T}\times\mathbb R)} &\lesssim \|\mathcal J^s(uv)\|_{L^2}\|w\|_{L^4}+\|uv\|_{L^2}\|\mathcal J^sw\|_{L^4}\\
		\nonumber
		&\lesssim \|\mathcal J^su\|_{L^4}\|v\|_{L^4}\|w\|_{L^4}+\|u\|_{L^4}\|\mathcal J^sv\|_{L^4}\|w\|_{L^4}\\
		\label{d}
		&\qquad \qquad \qquad \qquad \qquad \quad \,\, +\|u\|_{L^4}\|v\|_{L^4}\|\mathcal J^sw\|_{L^4}.
	\end{align}
	
	\begin{lem}[\cite{ET-book} Lemma 1.11] \label{Js}
		Let $s\ge0$ and $1<p_i,q_i,r<\infty$ for $i=1,2$. Then 
		$$\|\mathcal J^s(fg)\|_{L^r}\lesssim\|f\|_{L^{p_1}}\|\mathcal J^sg\|_{L^{q_1}}+\|\mathcal J^sf\|_{L^{p_2}}\|g\|_{L^{q_2}}$$ where  $$\frac1r=\frac1{p_1}+\frac1{q_1}=\frac1{p_2}+\frac1{q_2}.$$
	\end{lem}
	
	Finally the $L^4$-norms in \eqref{c} and \eqref{d} are easily bounded as desired, using the following lemma with $b=3/8$ and the embedding $X^{s,b}\subset X^{0,b}$ for $s\ge0$:
	
	\begin{lem}[\cite{MT} Proposition 2.4]\label{L4}
		Let $b>1/3$. Then 
		$$\|f\|_{L^4(\mathbb T\times \mathbb R)}\le C\|f\|_{X^{0,b}}$$
		where the constant $C$ depends only on $b$.
	\end{lem}
	
	\section{Proofs of Theorems \ref{Tal} and \ref{t:dim}}\label{sec3-0}
	In this section we prove Theorems \ref{Tal} and \ref{t:dim} in the time interval $[-\delta,\delta]$ in which solutions are guaranteed to exist by Theorem \ref{wellp}.
	
	Before we prove the theorems in detail, we shall present some preliminaries.
	We first apply the Fourier transform to \eqref{TLL} to write 
	\begin{equation}\label{fourier}
		\begin{cases}	
			\partial_t \widehat u(k)=-ik^3\widehat u(k)-ik^2\widehat u(k)-\widehat u(k)+i\widehat{|u|^2u}(k)+\widehat {u_0}(k),\\
			\widehat u(0,k)=\widehat {u_0}(k),
		\end{cases}	
	\end{equation}
	and decompose the cubic nonlinear term $\widehat{|u|^2u}(k)$  as 
	\begin{align}\label{res}
		\widehat{|u|^2u}(k)&=\sum_{k_1, k_2}\hat{u}(k_1)\bar{\hat{u}}(k_2)\hat{u}(k-k_1+k_2)\nonumber\\
		&=2\|\hat{u}\|_{l_k^2}^2\hat{u}(k)-|\hat{u}(k)|^2\hat{u}(k)+\sum_{\substack{k_1\neq k\\k_2\neq k_1}}\hat{u}(k_1)\bar{\hat{u}}(k_2)\hat{u}(k-k_1+k_2)
	\end{align}
	by which the equation \eqref{fourier} is rearranged as
	\begin{equation*}
		\begin{cases}
			\partial_t \widehat u+(ik^3+ik^2+1-2i\|\widehat u\|_{l_k^2}^2)\widehat u=i\widehat{\rho(u)}+i\widehat{R(u)}+\widehat u_0,\\
			\widehat u(0,k)=\widehat {u_0}(k),
		\end{cases}	
	\end{equation*}
	where
	\begin{equation}\label{rho}
		\widehat{\rho(u)}(k)=-|\hat{u}(k)|^2\hat{u}(k),
	\end{equation}
	\begin{equation}
		\label{Rr}
		\widehat{R(u)}(k)=\sum_{\substack{k_1\neq k\\k_2\neq k_1}}\hat{u}(k_1)\bar{\hat{u}}(k_2)\hat{u}(k-k_1+k_2).
	\end{equation}
	Regarding the term $i{\rho(u)}+i{R(u)}+u_0$ as a source term,
	the solution can be then written as
	\begin{equation}\label{TLL-D}
		u(t,x)=e^{\partial_x^3t+i\partial_x^2t-t+4\pi i\int_0^t\|u\|_{L_x^2}^2ds}u_0+\mathcal{N}(t,x),
	\end{equation}
	where $$\mathcal N(t,x)=i\int_0^t e^{(\partial_x^3+i\partial_x^2-1)(t-t')} e^{4\pi i\int_{t'}^t\|u\|_{L_x^2}^2ds}(i\rho(u)+iR(u)+u_0)dt'$$
	satisfies the following smoothing property that is the key ingredient in the proofs and will be proved in Section \ref{sec3}.
	\begin{prop}\label{smoothing}
		Let $s>0$ and $0<a<\min\left\{2s,1\right\}$.	
		If $u_0\in H^s (\mathbb{T})$,
		then we have
		\begin{equation}\label{nonsmo}
		\mathcal{N}(t,x)\in C([-\delta,\delta]; H_x^{s+a}(\mathbb T))
  \end{equation}
		with $\delta>0$ obtained in Theorem \ref{wellp}. 
	\end{prop}

	\subsection{Proof of Theorem \ref{Tal}}
	We first show that $\mathcal N(t,\cdot)$ is a continuous function.
	Since $u_0\in BV$, we have $u_0\in H^{\frac12-}$.
	Then, by Proposition \ref{smoothing} with $s=\frac12-$, we conclude that 
	\begin{equation*}
		\mathcal N(t,x) \in C([-\delta,\delta];H_x^{\frac32-}(\mathbb T)).
	\end{equation*}
	From the Sobolev embedding
	\begin{equation}\label{e:sob-emb}
		H^s \hookrightarrow C^{s-\frac {1}{2}} \quad \text{for} \quad s>1/2,
	\end{equation} 
	it follows that 
	\begin{equation}\label{e:joint-conti}
		\mathcal N(t,x) \in C([-\delta,\delta];C_x^{1-}(\mathbb T)),
	\end{equation}
	and hence $\mathcal N(t,\cdot)$ is continuous on $[-\delta,\delta]\times \mathbb{T}$.
	
	For the homogeneous part $e^{\partial_x^3t+i\partial_x^2t-t+4\pi i\int_0^t\|u\|_{L^2}^2dt'}u_0$, we shall make use of the following known result due to Oskolkov \cite[Proposition 14 and p. 390]{O}:
	\begin{prop}\label{Tal1}
		Suppose that  $u_0\in BV$. 
		\begin{enumerate}
			[leftmargin=0.7cm]
			\item[$(\romannum{1})$] If $t\notin\pi\mathbb{Q}$, then $e^{\partial_x^3t+i\partial_x^2t}u_0$ is a continuous function of $x$. If $t\in\pi\mathbb{Q}$ and $u_0$ has at least one discontinuity on $\mathbb T$, then $e^{\partial_x^3t+i\partial_x^2t}u_0$ necessarily contains discontinuities.
			\item[$(\romannum{2})$] If $u_0$  is continuous, then $e^{\partial_x^3t+i\partial_x^2t}u_0$ is jointly continuous in temporal and spatial variables. 
		\end{enumerate}
	\end{prop}
	\begin{rem}\label{rem1}
		It is known in \cite[Theorem 2.14]{ET-book} that
		$e^{\partial_x^3t+i\partial_x^2t}u_0$ is a linear sum of finitely many translates of $u_0$ if $t\in\pi\mathbb Q$. 
		Hence, in Proposition \ref{Tal1} $(\romannum{1})$, $e^{\partial_x^3t+i\partial_x^2t}u_0\in BV$ if $t\in \pi \mathbb {Q}$, so it contains at most countable discontinuities.
	\end{rem}
	Now for $t\notin \pi \mathbb Q$, it follows by Proposition \ref{Tal1} $(\romannum{1})$ that 
	\begin{equation}\label{e:evolution}
		e^{\partial_x^3t+i\partial_x^2t-t+4\pi i\int_0^t\|u\|_{L^2}^2ds}u_0 = e^{-t+4\pi i\int_0^t\|u\|_{L^2}^2ds}e^{\partial_x^3t+i\partial_x^2t}u_0 
	\end{equation}
	is continuous. 
	Combined with \eqref{e:joint-conti}, the solution \eqref{TLL-D} becomes continuous.
	
	If $t\in \pi \mathbb Q$ and $u_0$ is discontinuous, then it follows from Proposition \ref{Tal1} $(\romannum{1})$ and Remark \ref{rem1} that the evolution \eqref{e:evolution} is of bounded variation and contains at most countable discontinuities. Combining this with \eqref{e:joint-conti}, the solution \eqref{TLL-D} is a discontinuous bounded function with at most countable discontinuities. 
	
	Finally, if $u_0$ is continuous, so is the evolution \eqref{e:evolution} by Proposition \ref{Tal1} $(\romannum{2})$. Combining this with \eqref{e:joint-conti}, we therefore see that the solution \eqref{TLL-D} is jointly continuous on $[-\delta,\delta]\times \mathbb T$.
	
	\subsection{Proof of Theorem \ref{t:dim}}	
	We begin by introducing the Besov space and its properties that we need.
	Let $\phi \in C^\infty_0([-2,-\frac12] \cup [\frac12,2])$ be such that $\sum_{j\in\mathbb Z}\phi(2^{-j}t)=1$ for $t\in\mathbb R\setminus\{0\}$, and let $\phi_0(t):=1-\sum_{j\ge1}\phi(2^{-j}t)$. We denote by $P_j$ the projections defined by  
	\[	P_0f(x):=\sum_{k\in \mathbb Z} \phi_0(k)\widehat f(k)e^{ikx}, \quad P_j f(x):= \sum_{k\in \mathbb Z} \phi(2^{-j}k)\widehat f(k)e^{ikx}, \ \ j\ge 1.	\] 
	For $1\le p\le\infty$ and $s\ge 0$, the inhomogeneous Besov space $B^s_{p,\infty}$ on $\mathbb T$ is a Banach space of functions equipped with the norm 
	\begin{equation*}
		\|f\|_{B_{p,\infty}^s} := 
		\sup_{j\ge 0} 2^{sj}\|P_jf\|_{L^p(\mathbb{T})}.
	\end{equation*}
	
	In order to calculate the upper Minkowski dimension $D$ of the solution graph, we exploit the following basic results in geometric measure theory:
	\begin{lem}[\cite{F}]\label{l:upper}
		Let $0\le \alpha\le 1$. If $f\colon \mathbb T\to \mathbb R$ is in $C^\alpha$, then the upper Minkowski dimension of the graph of $f$ is at most $2-\alpha$.  
	\end{lem}
	\begin{lem}[\cite{DJ}]\label{FDL}
		Let $0<s<1$. If $f\colon\mathbb{T}\to \mathbb{R}$ is continuous and $f\notin B_{1,\infty}^{s+}$, then the upper Minkowski dimension of the graph of $f$ is at least $2-s$.
	\end{lem}
	
	By Lemma \ref{l:upper}, we first conclude that the upper bound of the upper Minkowski dimension $D$ of both  $\mathrm{Re}\,u(t,\cdot)$ and $\mathrm{Im}\,u(t,\cdot)$ is $7/4$,
	if we show 
	\begin{equation}\label{e}
		u(t,\cdot)\in C^{\frac14-} \quad\ \text{for almost all}\,\,\, t\in[-\delta,\delta]\setminus\pi\mathbb{Q}.
	\end{equation}
	To show \eqref{e} for the Duhamel term 
	$\mathcal N(t,x)$ in \eqref{TLL-D}, we use Proposition \ref{smoothing};
	since $u_0\in H^{\sigma_0-}$ and $1/2\leq \sigma_0<5/8$ (from the definition of $\sigma_0$), we have
	\begin{equation}\label{g}
		\mathcal N(t,\cdot)\in C([-\delta,\delta];H_x^{\frac32-}(\mathbb T))
	\end{equation} 
	by Proposition \ref{smoothing}, and moreover
	\begin{equation}\label{Nsmoot}
		\mathcal N(t,\cdot)\in C([-\delta,\delta];C_x^{1-}(\mathbb T))
	\end{equation}
	by the Sobolev embedding \eqref{e:sob-emb}.
	For the homogeneous term in \eqref{TLL-D}, we apply the following proposition (see \cite[Theorem 2.16]{ET-book}) with degree $d=3$ to get
	\begin{equation}\label{lin:reg}
		e^{\partial_x^3t+i\partial_x^2t-t+4\pi i\int_0^t\|u\|_{L^2}^2ds}u_0\in C^{\frac14-}\setminus B_{1,\infty}^{(2s_0-\frac14)+},
	\end{equation}
	and thus \eqref{e} follows by combining \eqref{Nsmoot} and \eqref{lin:reg}.
	
	\begin{prop}\label{linear}
		Let $P$ be a polynomial of degree $d$ with integer coefficients, and $P(0)=0$. Suppose that $g:\mathbb{T}\to\mathbb{C}$ is of bounded variation. Then for almost every $t\notin\pi\mathbb Q$, $e^{itP(-i\partial_x)}g\in C^\alpha(\mathbb{T})$ for any $0\le\alpha<2^{1-d}$. Moreover, when $P$ is not an odd polynomial, if in addition $g\notin H^{s_0+}$ for some $s_0\in[\frac12,\frac12+2^{-d})$, then for almost all $t\notin\pi\mathbb Q$ both $\mathrm{Re}( e^{itP(-i\partial_x)}g)$ and $\mathrm{Im}( e^{itP(-i\partial_x)}g)$ do not belong to $B_{1,\infty}^{(2s_0-2^{1-d})+}$.
	\end{prop}
	
	Next we show that the lower bound of $D$ is $9/4-2\sigma_0$.
	Combining \eqref{lin:reg} and Proposition \ref{linear} with $d=3$ and $\alpha=\frac14-$, we have 
	$$\mathrm{Re}( e^{\partial_x^3t-i\partial_x^2t-t-4\pi i\int_0^t\|u\|_{L^2}^2ds}u_0),\,
	\mathrm{Im}( e^{\partial_x^3t-i\partial_x^2t-t-4\pi i\int_0^t\|u\|_{L^2}^2ds}u_0)
	\notin B^{(2\sigma_0-\frac14)+}_{1,\infty}$$
	for almost all $t\notin\pi\mathbb Q$.
	But, $\mathcal N(t,\cdot)\in B_{1,\infty}^{\frac32-}$ by \eqref{g} and the embedding $H^\alpha \hookrightarrow B_{1,\infty}^{\alpha}$ for $\alpha\ge0$. 
	Hence, $\mathrm{Re}\,u(t,\cdot),\,\mathrm{Im}\,u(t,\cdot) \notin B_{1,\infty}^{(2\sigma_0-\frac14)+}$ because   
	$B^{\frac32-}_{1,\infty}   \hookrightarrow B^{(2\sigma_0-\frac14)+}_{1,\infty}$ from $2\sigma_0-\frac14<\frac32$.
	Consequently, by Lemma \ref{FDL}, we conclude that both the real and imaginary parts of $u(t,\cdot)$ have upper Minkowski dimension at least $2-(2\sigma_0-\frac14)=9/4-2\sigma_0$.

\section{Smoothing estimate}\label{sec3}
This section is devoted to proving Proposition \ref{smoothing}. 
From \eqref{res}, \eqref{rho} and \eqref{Rr}, we first notice that 
$|u|^2u=4\pi\|u\|_{L^2}^2u+\rho(u)+R(u)$.
Substituting this and the transformation\footnote{The transformation is used to eliminate the problematic exponential factor from $\mathcal{N}(t,x)$, but dealing with the forcing term can still be somewhat cumbersome since the factor remains in front of the forcing term in \eqref{LLE-D}.}
\begin{equation}\label{trans}
	u(t,x)=e^{4\pi i\int_0^t\|v(t_1)\|_{L^2}^2dt_1}v(t,x)
\end{equation}
into \eqref{TLL}, we then see  
\begin{equation*}
	\begin{cases}
		\partial_t v = \partial_x^3 v + i \partial_x^2 v - v + e^{-4\pi i \int_0^t \|v\|_{L^2}^2 dt_1} u_0 + i(\rho(v)+R(v)), \\
		v(0,x)=u_0(x),
	\end{cases}
\end{equation*}
and by Duhamel's formula
\begin{align}\label{LLE-D}
v(t,&x)=e^{(\partial_x^3 + i \partial_x^2 -1)t}u_0\nonumber\\
&+\int_0^te^{(\partial_x^3 + i \partial_x^2 -1)(t-t_1)}\big(e^{-4\pi i\int_0^{t_1}\|v\|_{L^2}^2dt_2}u_0+i(\rho(v)+R(v))(t_1)\big)\, dt_1.
\end{align}
Since
$\|u(t)\|_{L^2} = \|v(t)\|_{L^2}$, 
from \eqref{trans} and \eqref{TLL-D} we also see
\begin{equation*}
	v(t,x)=e^{-4\pi i\int_0^t\|u\|_{L^2}^2dt_1}u(t,x)=e^{(\partial_x^3+i\partial_x^2-1)t}u_0+e^{-4\pi i\int_0^t\|u\|_{L^2}^2dt_1}\mathcal{N}(t,x).
\end{equation*}
By comparing this to \eqref{LLE-D} and noticing 
$$\big\|e^{-4\pi i\int_0^t\|u\|_{L^2}^2dt_1}\mathcal{N}(t,x)\big\|_{H^s}=\big\|\mathcal{N}(t,x)\big\|_{H^s},$$
it suffices to prove the smoothing property \eqref{nonsmo} for the Duhamel term in \eqref{LLE-D} instead of $\mathcal N(t,x)$.
To do so, we need to make use of the following multilinear estimates for $\rho(v)$, $R(v)$ and $e^{-4\pi i \int_0^{t}\|v\|_{L^2}^2 dt_1} u_0$:
\begin{lem}\label{mul}
	Let $\frac12<b\le 1$ and $s> \frac{b}{2}-\frac14$. 
	For $a\le\min\{2s+1-2b, 2-2b\}$, we have
	\begin{equation}\label{ff2}
		\|\rho(v)\|_{H^{s+a}}\lesssim \|v\|_{H^s}^3,
	\end{equation}
	\begin{equation}\label{ff}
		\left\|e^{-4\pi i\int_0^{t}\|v(t')\|_{L^2}^2dt'}u_0\right\|_{X_\delta^{s+a,b-1}}\lesssim \|u_0\|_{H^s},
	\end{equation}
	and
	\begin{equation}\label{ff3}
		\|R(v)\|_{X_\delta^{s+a,b-1}}\lesssim \|v\|_{X_{\delta}^{s,b}}^3.
	\end{equation}
\end{lem}

We will postpone the proof of this lemma to the next section and continue with the proof of Proposition \ref{smoothing}.
Using the embedding 
\begin{equation}\label{emb}
	X_{\delta}^{s,b} \subset C_t H_x^s([-\delta,\delta]\times\mathbb T), \quad b>1/2,
\end{equation}
the estimate \eqref{inhomo}, and Lemma \ref{mul} in turn, 
we see that 
\begin{align}\label{duhamel}
\nonumber
&\Big\|\int_0^t e^{(\partial_x^3+i\partial_x^2-1)(t-t_1)}\big(e^{-4\pi i\int_0^{t_1}\|v\|_{L^2}^2dt_2}u_0+i(\rho(v)+R(v))(t_1)\big)dt_1\Big\|_{H^{s+a}}\\
\nonumber
&\lesssim\int_0^t\|\rho(v)(t_1)\|_{H^{s+a}}dt_1\\
\nonumber
&\qquad +\Big\|\int_0^t e^{(\partial_x^3+i\partial_x^2-1)(t-t_1)}\big(e^{-4\pi i\int_0^{t_1}\|v\|_{L^2}^2dt_2}u_0+iR(v)(t_1)\big)dt_1\Big\|_{X_\delta^{s+a,b}}\\
\nonumber
&\lesssim \big\|e^{-4\pi i\int_0^{t}\|v\|_{L^2}^2dt_1}u_0\big\|_{X_\delta^{s+a,b-1}}+\|R(v)\|_{X_\delta^{s+a,b-1}}+\int_0^t\|\rho(v)(t_1)\|_{H^{s+a}}dt_1\\
&\lesssim\|u_0\|_{H^s}+\|v\|_{X_\delta^{s,b}}^3+\delta\|v\|_{H^s}^3
\end{align}
under the assumptions on $b,s,a$ in Lemma \ref{mul}.
By the embedding \eqref{emb} the last two terms in \eqref{duhamel} are bounded by $C\|v\|_{X_\delta^{s,b}}^3$. 
Therefore, if we show that 
\begin{equation}\label{vu0}
\|v\|_{X_\delta^{s,b}}\lesssim\|u\|_{X_\delta^{s,b}}
\end{equation} 
for $b,s,\delta$ given in Theorem \ref{wellp}, the desired property \eqref{nonsmo}
is now proven by virtue of \eqref{prop2}. 
As $b$ approaches $1/2$, the ranges of $s$ and $a$ in Lemma \ref{mul} become the widest, and the ranges used in Proposition \ref{smoothing} follow accordingly.

It remains to prove \eqref{vu0}.
Let $\eta(t) \in C_0^{\infty}(\mathbb{R})$ be a smooth function compactly supported in $(-2\delta,2\delta)$ with $\eta=1$ on $[-\delta,\delta]$. 
By the definition of $X_\delta^{s,b}$ and \eqref{trans}, it is enough to show that
\begin{equation}\label{vu00} 
\Big\|\eta(t) e^{-4\pi i\int_0^t\|u\|_{L^2}^2\,dt_1} w\Big\|_{X^{s,b}}\lesssim \| w\|_{X^{s,b}}
\end{equation}
for all $w\in X^{s,b}$ such that $w=u$ on $[-\delta,\delta]$.
Set $\zeta(t):=\eta(t)e^{-4\pi i\int_0^t\|u\|_{L^2}^2\,dt_1}$.
Then,
\begin{align*}
\|\zeta(t)w\|_{X^{s,b}}^2&=\sum_k\int_{\mathbb{R}}\lgl k\rgl^{2s}\lgl\tau+k^3+k^2\rgl^{2b}\left|\int_{\mathbb{R}}\widehat{\zeta}(\tau-\tau_1) \, \widetilde{w}(\tau_1,k) \,d\tau_1\right|^2d\tau\\
&=\sum_k\lgl k\rgl^{2s}\int_{\mathbb{R}}\left|\int_{\mathbb{R}}\lgl\tau+k^3+k^2\rgl^{b}\widehat{\zeta}(\tau-\tau_1)\widetilde{w}(\tau_1,k)d\tau_1\right|^2d\tau.
\end{align*}
Using the inequality $\lgl\tau+k^3+k^2\rgl\lesssim\lgl\tau-\tau_1\rgl\lgl\tau_1+k^3+k^2\rgl$ and Young's convolution inequality, we have
\begin{align*}
\|\zeta(t)w\|_{X^{s,b}}^2&\lesssim\big\|\lgl\tau\rgl^b\widehat{\zeta}(\tau)\big\|_{L_\tau^1}^2\big\|\lgl k\rgl^{s}\lgl\tau+k^3+k^2\rgl^{b}\, \widetilde{w}(\tau,k)\big\|_{L_\tau^2 l_k^2}^2\\
&=\big\|\lgl\tau\rgl^b\widehat{\zeta}(\tau)\big\|_{L_\tau^1}^2\|w\|_{X^{s,b}}^2.
\end{align*}
Now \eqref{vu00} follows if we show that the $L_{\tau}^1$-norm here is finite. For this we note that
\begin{equation}\label{c1}
    \zeta(t)\in C^2(\mathbb R).
\end{equation}
Since $\zeta(t)$ is compactly supported, this implies 
$|\widehat \zeta(\tau)|\lesssim \lgl \tau\rgl^{-2}$ (see \cite{W}),
so that 
$$\big\|\lgl\tau\rgl^b\widehat{\zeta}(\tau)\big\|_{L_\tau^1}\lesssim \int_{\mathbb R}\lgl \tau\rgl^{b-2}\, d\tau<\infty$$
whenever $b<1$.
Finally, the continuity \eqref{c1} follows from $u\in C_t^1([-2\delta,2\delta];L^2(\mathbb T))$.
By the equation \eqref{TLL}, we indeed see
\begin{align*}
\frac{d}{dt}\|u\|_{L^2}^2=\frac{d}{dt}\int_{\mathbb T}u\bar u\, dx&=2 \Re\int_{\mathbb T}\bar u\partial_t udx\\
&=2\Re\int_{\mathbb T}\partial_x^3 u\bar u + i \partial_x^2 u\bar u - u\bar u +i |u|^4 +u_0\bar u  \,dx,
\end{align*}
and by integration by parts
$$\frac{d}{dt}\|u\|_{L^2}^2=-2\|u(t)\|_{L^2}^2+2\Re\int_{\mathbb T}u_0\bar u\, dx.$$
The right side (and similarly the left side) is continuous because Theorem \ref{wellp} guarantees that $u\in C_t([-2\delta,2\delta];L^2(\mathbb T))$.

\section{Proof of Lemma \ref{mul}}\label{sec5}
In this final section, we will prove Lemma \ref{mul}, which provides estimates for $e^{-4\pi i \int_0^{t}\|v\|_{L^2}^2 dt'} u_0$, $\rho$, and $R$. We will first focus on $e^{-4\pi i \int_0^{t}\|v\|_{L^2}^2 dt'} u_0$ and $\rho$, which are relatively simpler to handle.

\subsection{Proofs of \eqref{ff2} and \eqref{ff}}
	For $a\leq 2s$, using the inclusion property for $l^p$-spaces we obtain \eqref{ff2} as
	\begin{equation*}
		\|\rho(v)\|_{H^{s+a}}=\|\langle k \rangle^{s+a} |\hat v(k)|^3\|_{l_k^2}=\|\langle k \rangle^{\frac{s+a}{3}} \hat v(k)\|^3_{l_k^6}
		\lesssim\|\langle k\rangle^s \hat v(k)\|_{l_k^2}^3.
	\end{equation*}
 
	The proof of \eqref{ff} is similar to that of \eqref{force} in Proposition \ref{prop3}. 
	Since $\eta(t)e^{-4\pi i\int_0^{t}\|v\|_{L^2}^2dt'}\in C_t^1$ is compactly supported, we first see $$\Big|\reallywidehat{\eta(t)e^{-4\pi i\int_0^{t}\|v\|_{L^2}^2dt'}}(\tau)\Big|\lesssim\langle\tau\rangle^{-1}$$
 which implies 
	\begin{align*}
\big\|e^{-4\pi i\int_0^{t}\|v\|_{L^2}^2dt'}u_0\big\|^2_{X_\delta^{s+a,b-1}}
& \le\big\|\eta(t)e^{-4\pi i\int_0^{t}\|v\|_{L^2}^2dt'}u_0\big\|^2_{X^{s+a,b-1}}\\
&\lesssim\sum_{k\in\mathbb{Z}} \int_{\mathbb R}\langle k\rangle^{2s+2a}\langle\tau+ k^3+ k^2\rangle^{2b-2}\langle \tau\rangle^{-2}|\widehat{u_0}(k)|^2d\tau.
\end{align*}
Applying Lemma \ref{lem3} with $\beta=2$ and $\gamma=2-2b$ to the integration in $\tau$, we now get
	\begin{align*}
\big\|e^{-4\pi i\int_0^{t}\|v\|_{L^2}^2dt'}u_0\big\|^2_{X_\delta^{s+a,b-1}}
&\lesssim\sum_{k\in\mathbb{Z}} \langle k\rangle^{2s+2a}\langle k^3+k^2\rangle^{2b-2}|\widehat{u_0}(k)|^2\\
&\lesssim\sum_{k\in\mathbb{Z}} \langle k\rangle^{2s+2a+6b-6}|\widehat{u_0}(k)|^2\\
&\lesssim\|u_0\|_{H^s}^2
\end{align*}
provided $a\leq 3-3b$.

\subsection{Proof of \eqref{ff3}}
Lastly, we handle $R$. By the definition of the restricted space $X_{\delta}^{s,b}$, we may show \eqref{ff3} with $X^{s,b}$ rather than $X_{\delta}^{s,b}$.
First set
	\begin{align*}	
		f_i(\tau_i,k_i)&:=\langle k_i \rangle^s\langle \tau_i+k_i^3+k_i^2\rangle^b|\tilde v(\tau_i, k_i)|, \quad i=1,2,3, \\
		M(k_1,k_2,k_3,\tau_1,\tau_2,\tau_3)&:=
		\frac{\langle k\rangle^{s+a}\langle \tau +k^3+ k^2 \rangle^{b-1} \langle k_1 \rangle^{-s}\langle k_2 \rangle^{-s}\langle k_3 \rangle^{-s}}{\langle\tau_1+k_1^3+ k_1^2\rangle^{b}\langle \tau_2+k_2^3+ k_2^2\rangle^{b}\langle \tau_3+k_3^3+ k_3^2\rangle^{b}}.
	\end{align*}
	Then we write  
	\begin{align}\nonumber
		&\big\|R(v)\big\|_{X^{s+a,b-1}}^2\\
		\nonumber
		&=\Big\|\langle k\rangle^{s+a}\langle \tau+k^3+ k^2\rangle^{b-1}\\
		\nonumber
		&\qquad\times \int_{\tau_1-\tau_2+\tau_3=\tau}\sum_{\substack{k_1-k_2+k_3=k\\k_1\neq k,\,k_2\neq k_1}}\tilde{v}(\tau_1,k_1)\bar{\tilde{v}}(\tau_2,k_2)\tilde{v}(\tau_3,k_3)d\tau_1d\tau_2\Big\|_{L_{\tau}^2l_k^2}^2\\
		\label{R}
		&=\Big\| \int_{\tau_1-\tau_2+\tau_3=\tau}\sum_{\substack{k_1-k_2+k_3=k\\k_1\neq k,\,k_2\neq k_1}}Mf_1(\tau_1,k_1)f_2(\tau_2,k_2)f_3(\tau_3,k_3)d\tau_1 d\tau_2\Big\|_{L_{\tau}^2l_k^2}^2.
	\end{align}
Applying the Cauchy-Schwarz inequality, \eqref{R} is bounded by
\begin{align*}
\sup_{\tau,k}&\int_{\tau_1-\tau_2+\tau_3=\tau}\sum_{\substack{k_1-k_2+k_3=k\\k_1\neq k,\,k_2\neq k_1}}M^2d\tau_1d\tau_2\\
&\times\bigg\|\int_{\tau_1-\tau_2+\tau_3=\tau}\sum_{k_1-k_2+k_3=k}|f_1f_2f_3|^2d\tau_1d\tau_2\bigg\|_{L_{\tau}^1l_k^1}.
\end{align*}
The $L^1$ norm here is bounded by  
Young's convolution inequality as 
	\begin{align*}
		\bigg\|
		\int_{\tau_1-\tau_2+\tau_3=\tau}\sum_{k_1-k_2+k_3=k}|f_1f_2f_3|^2d\tau_1d\tau_2\bigg\|_{L_{\tau}^1l_k^1}
&=	\|\big(f_1^2\ast_{\tau,k} f_2^2\ast_{\tau,k} f_3^2\big) (\tau,k)\|_{L_{\tau}^1l_k^1}\\
&\leq \||f|^2\|_{L_{\tau}^1l_k^1}^3
  =\|v\|_{X^{s,b}}^6.
	\end{align*}
All we have to do now is to show that 
	\begin{equation}\label{M}
		\sup_{\tau,k}\int_{\tau_1-\tau_2+\tau_3=\tau}\sum_{\substack{k_1-k_2+k_3=k\\k_1\neq k,\,k_2\neq k_1}}M(k_1,k_2,k_3,\tau_1,\tau_2,\tau_3)^2 d\tau_1d\tau_2
	\end{equation}
 is finite.

To bound the integration with respect to $\tau_2,\tau_1$ first, we use Lemma \ref{lem3} twice with $\beta=\gamma=2b>1$ as follows:
	\begin{align}
		\nonumber
		&\int_{\tau_1} \frac{1}{\langle \tau+k^3+k^2 \rangle^{2-2b}\langle \tau_1+k_1^3+k_2^2\rangle^{2b} }\\
		\nonumber
		&\quad\qquad\times \int_{\tau_2} \frac{1}{\langle \tau_2 + k_2^3+ k_2^2\rangle^{2b}\langle \tau-\tau_1+\tau_2 + k_3^3+ k_3^2\rangle^{2b}}d\tau_2d\tau_1\\
		\nonumber
		&\lesssim \int_{\tau_1}  \frac{1}{\langle \tau+k^3+k^2 \rangle^{2-2b}\langle \tau_1+k_1^3+k_2^2\rangle^{2b}\langle \tau_1-\tau +k_2^3+k_2^2-k_3^3-k_3^2 \rangle^{2b}} d\tau_1 \\
		\label{R1}
		&\lesssim \frac{1}{\langle \tau+k^3+k^2 \rangle^{2-2b}\langle \tau +k_1^3-k_2^3+k_3^3+k_1^2-k_2^2+k_3^2 \rangle^{2b}}.
	\end{align}
Then, by substituting $k_3=k-k_1+k_2$ into the right side of \eqref{R1} and using
the simple inequality $\langle \tau+m\rangle\langle \tau+n\rangle \gtrsim \langle m-n \rangle$, 
it is bounded by 
	\begin{equation*}
		\frac{1}{\langle (k-k_1)(k_1-k_2)(3k+3k_2+2)\rangle^{2-2b}}. 
	\end{equation*}
Since $(k-k_1)(k_1-k_2)(3k+3k_2+2)$ is non-zero and $|3(k+k_2)+2|\ge|k+k_2|$ for all $k,k_2\in\mathbb{Z}$, this is bounded again by
	\begin{equation*}
		\frac{1}{\langle k-k_1\rangle^{2-2b}\langle k_1-k_2\rangle^{2-2b}\langle k+k_2\rangle^{2-2b}}.
	\end{equation*}
Therefore, 
	\begin{equation}
\eqref{M}\lesssim \sup_k\sum_{\substack{k_1\neq k\\k_2\neq k_1}}\frac{\langle k\rangle^{2s+2a} \langle k_1\rangle^{-2s} \langle k_2 \rangle^{-2s}\langle k-k_1+k_2\rangle^{-2s} }{\langle k-k_1\rangle^{2-2b}\langle k_1-k_2\rangle^{2-2b}\langle k+k_2\rangle^{2-2b}}.\label{sum}
\end{equation}
	
Let us now bound the summation in \eqref{sum}.
	We first consider the case $s>\frac12$ and then divide the matter into three cases, $|k-k_1+k_2|\gtrsim|k|$, $|k_1|\gtrsim|k|$ and $|k_2|\gtrsim|k|$.
	
 \begin{flalign*}
		(\romannum{1})\quad |k-k_1+k_2|\gtrsim|k|&&
	\end{flalign*}
In this case, we replace $\langle k-k_1+k_2\rangle^{-2s}$ by
$\langle k\rangle^{-2s}$, and 
use the inequality $\langle k+k_2\rangle\langle k_2\rangle\gtrsim\langle k\rangle$ after multiplying $\langle k_2\rangle^{2-2b}$ to both numerator and denominator. Then, \eqref{sum} is bounded by
	\begin{align}\label{kkk2}
		\sum_{k_1,k_2}\frac{\langle k\rangle^{2a-2+2b} \langle k_1\rangle^{-2s} \langle k_2 \rangle^{-2s+2-2b}}{\langle k-k_1\rangle^{2-2b}\langle k_1-k_2\rangle^{2-2b}}.
	\end{align}
 
Applying Lemma \ref{lem3} to $k_2$-summation with $\beta=2-2b$ and $\gamma=2s-2+2b$, we obtain 
\begin{align*}
		\eqref{kkk2}\lesssim
		\begin{cases}
			\sum_{k_1}\dfrac{\langle k\rangle^{2a-2+2b}}{\langle k_1\rangle^{2s}\langle k-k_1\rangle^{2-2b}}\langle k_1\rangle^{-2s+2-2b}\phi_{2-2b}(k_1)\qquad s<2-2b,\\
			\sum_{k_1}\dfrac{\langle k\rangle^{2a-2+2b}}{\langle k_1\rangle^{2s}\langle k-k_1\rangle^{2-2b}}\langle k_1\rangle^{2b-2}\phi_{2s-2+2b}(k_1)\qquad \,\,\,\,2-2b\le s,
		\end{cases}
	\end{align*}
depending on the order relation between $2-2b$ and $2s-2+2b$. Since we are assuming $\frac12 <b\le 1$ and considering $s>\frac12$, the condition to guarantee Lemma \ref{lem3} 
(i.e., $\beta,\gamma\ge0$, $\beta+\gamma>1$) is satisfied.
We calculate $\phi_{2-2b}(k_1), \phi_{2s-2+2b}(k_1)$ depending on the situation to yield
	\begin{align*}
		\eqref{kkk2}\lesssim
		\begin{cases}
			\sum_{k_1}\dfrac {\langle k\rangle^{2a-2+2b}}{\langle k_1\rangle^{4s-1}\langle k-k_1\rangle^{2-2b}}\qquad\qquad\quad s<2-2b,\\
			\sum_{k_1}\dfrac{\langle k\rangle^{2a-2+2b}}{\langle k_1\rangle^{2s+2-2b}\langle k-k_1\rangle^{2-2b}}\quad \quad2-2b\le s,\ \frac32-b<s,\\
			\sum_{k_1}\dfrac{\langle k\rangle^{2a-2+2b}}{\langle k_1\rangle^{2s+2-2b-\epsilon}\langle k-k_1\rangle^{2-2b}} \quad2-2b\le s,\ \frac32-b=s,\\
			\sum_{k_1}\dfrac{\langle k\rangle^{2a-2+2b}}{\langle k_1\rangle^{4s-1}\langle k-k_1\rangle^{2-2b}}\, \qquad \quad2-2b\le s,\ \frac32-b>s.
		\end{cases}
	\end{align*}
In the third case here, $\log(1+\langle k_1\rangle)$ term appears but is replaced by $\langle k_1\rangle^{\epsilon}$ for sufficiently small $\epsilon>0$. 

Finally, we apply Lemma \ref{lem3} again to $k_1$-summation with $\gamma=2-2b$ and $\beta=4s-1, 2s+2-2b, 2s+2-2b-\epsilon$, and $4s-1$ respectively for each case:
	\begin{align*}
		\eqref{kkk2}\lesssim
		\begin{cases}
			\dfrac{\langle k\rangle^{2a-2+2b}}{\langle k\rangle^{2-2b}}\phi_{4s-1}(k)\qquad \qquad \quad s<2-2b,\\
			\dfrac{\langle k\rangle^{2a-2+2b}}{\langle k\rangle^{2-2b}}\phi_{2s+2-2b}(k)\quad \quad2-2b\le s,\ \frac32-b<s,\\
			\dfrac{\langle k\rangle^{2a-2+2b}}{\langle k\rangle^{2-2b}}\phi_{2s+2-2b-\epsilon}(k)\quad2-2b\le s,\ \frac32-b=s,\\
			\dfrac{\langle k\rangle^{2a-2+2b}}{\langle k\rangle^{2-2b}}\phi_{4s-1}(k)\,\qquad \quad2-2b\le s,\ \frac32-b>s.
		\end{cases}
	\end{align*}
From the assumptions $s>\frac12$ and $\frac12<b\le 1$, we have $4s-1>1$ and $2s+2-2b>1$, thus $\beta+\gamma>1$, $\beta\ge\gamma\ge 0$ and $\beta>1$ for each case.
Since $\beta>1$, all the functions $\phi_\beta(k)$ above are bounded by some constant, 
and therefore 
$\eqref{kkk2}\lesssim \langle k\rangle^{2a-4+4b}$ which is finitely bounded when $a\le 2-2b$.

	\begin{flalign*}
		(\romannum{2})\quad |k_1|\gtrsim|k|&&
	\end{flalign*}
	In this case, we bound the sum \eqref{sum} by
	\begin{align}\label{k1}
		\langle k\rangle^{2a}\sum_{k_1,k_2}\frac{ \langle k_2\rangle^{-2s} \langle k-k_1+k_2 \rangle^{-2s}}{\langle k-k_1\rangle^{2-2b}\langle k_1-k_2\rangle^{2-2b}\langle k+k_2\rangle^{2-2b}}
	\end{align}	
 replacing $\langle k_1\rangle^{-2s}$ by $\langle k\rangle^{-2s}$.
 We now multiply $\langle k_2\rangle^{2-2b}$ to both numerator and denominator and replace variables by $n_1=k-k_1+k_2$ and $n_2=k_2$. Then, \eqref{k1} becomes  
	\begin{align*}
		\sum_{n_1,n_2}\frac{\langle k\rangle^{2a-2+2b} \langle n_1\rangle^{-2s}\langle n_2 \rangle^{-2s+2-2b} }{\langle n_1-n_2\rangle^{2-2b}\langle k-n_1\rangle^{2-2b}}.
	\end{align*}
Note that this is equivalent to \eqref{kkk2} and we are done.
	
	\begin{flalign*}
		(\romannum{3})\quad |k_2|\gtrsim|k|&&
	\end{flalign*}
	Similarly as before, the sum \eqref{sum} is bounded by
	\begin{align}\label{k2}
		\langle k\rangle^{2a}\sum_{k_1,k_2}\frac{ \langle k_1\rangle^{-2s} \langle k-k_1+k_2 \rangle^{-2s}}{\langle k-k_1\rangle^{2-2b}\langle k_1-k_2\rangle^{2-2b}\langle k+k_2\rangle^{2-2b}}.
	\end{align}
We first eliminate $\langle k+k_2\rangle^{2-2b}$ in the denominator by some constant $C$ and replace variables by $n_1=k_1$ and $n_2=k-k_1+k_2$.  Then, \eqref{k2} becomes
	\begin{align}\label{k22}
		\sum_{n_1,n_2}\frac{\langle k\rangle^{2a} \langle n_1\rangle^{-2s}\langle n_2 \rangle^{-2s} }{\langle k-n_1\rangle^{2-2b}\langle k-n_2\rangle^{2-2b}}.
	\end{align}
	Note that $0\le2-2b<1<2s$ from the assumptions. Applying Lemma \ref{lem3} to $n_1$ and $n_2$ respectively with $\beta=2s$ and $\gamma=2-2b$, \eqref{k22} is finally bounded by $\lgl k\rgl^{2a-4+4b}$ which is finitely bounded when $a\le 2-2b$.

 \

 By considering the cases $(\romannum{1})$, $(\romannum{2})$ and $(\romannum{3})$ in total, 
we see that \eqref{ff3} holds for $a\le 2-2b$ when $s>\frac12$.
In the case when $\frac{b}{2}-\frac14<s\le\frac12$, \eqref{ff3} is obviously proved in the same way as in the case $s>\frac12$. The only difference is that we bound the term $\frac{1}{\langle k+k_2\rangle^{2-2b}}$ by some constant trivially. Indeed, \eqref{ff3} holds for $a\le\min\{2s+1-2b,2-2b\}$ when $\frac{b}{2}-\frac14<s\le\frac12$. We omit the details.
The proof for \eqref{ff3} is complete.

\section*{}
\subsection*{Funding}
This research was supported by the POSCO Science Fellowship of POSCO TJ Park Foundation (S. Kim), and by NRF-2022R1A2C1011312 (I. Seo).

\subsection*{Disclosure statement}
The authors report that there are no competing interests to declare.

\subsection*{Data availability statement}
Data sharing is not applicable to this article as no datasets were generated or analyzed during the current study.


\end{document}